\documentclass{amsart}
\usepackage{graphicx}
\usepackage{amssymb,amscd,amsthm,amsxtra}
\usepackage{latexsym}
\usepackage{epsfig}
\usepackage{cite}

\numberwithin{equation}{section}

\newtheorem{theorem}{Theorem}[section]
\newtheorem{lemma}{Lemma}[section]
\newtheorem{proposition}{Proposition}[section]
\newtheorem{corollary}{Corollary}[section]
\newtheorem{definition}[theorem]{Definition}

\newtheorem{remark}{Remark}

\def\ds{\displaystyle}

\begin{document}

\title[Analyticity of free boundaries]{On higher order boundary Harnack
and analyticity of free boundaries}
\author{Chilin Zhang}
\address{Department of Mathematics, Columbia University, New York, NY 10027}\email{cz2584@columbia.edu}
\begin{abstract}
We establish a $C^{1,\alpha}$ Schauder estimate of a non-standard degenerate elliptic equation and use it to give another proof of the higher order boundary Harnack inequality. As an application, we obtain the analyticity of the free boundary in the classical obstacle problem based on iterating the boundary Harnack principle.
\end{abstract}
\maketitle

\section{Introduction}
\subsection{Classical boundary Harnack} Let
\begin{equation*}
    \Omega_{1}=\{\Gamma(x')<x_{n}<\Gamma(x')+1\}\cap\{|x'|<1\},\quad\Gamma(0)=0
\end{equation*}
be a $C^{0,1}$ domain, and $u_{1},u_{2}>0$ two continuous functions in $\Omega_{1}$ so that
\begin{equation}
    Lu_{1}=Lu_{2}=0,\quad u_{1}\big|_{\Gamma}=u_{2}\big|_{\Gamma}=0,
\end{equation}
with
\begin{equation*}
    L=tr(A\cdot\nabla^{2})\quad\mbox{or}\quad L=div(A\cdot\nabla)
\end{equation*}
a uniformly elliptic operator with measurable coefficients. The boundary Harnack inequality studies the regularity properties of the ratio $\ds w=\frac{u_{1}}{u_{2}}$ near the boundary $\Gamma$.

The simplest version of the boundary Harnack is due to Kemper\cite{Kemper72}, where he proved that if $L=\Delta$, then
\begin{equation*}
    \frac{u_{1}}{u_{2}}(x)\leq C\cdot\frac{u_{1}}{u_{2}}(\frac{e_{n}}{2})\quad\mbox{in}\quad\Omega_{1/2}=\{\Gamma(x')<x_{n}<\Gamma(x')+1/2\}\cap\{|x'|<1/2\},
\end{equation*}
which implies that the ratio $w$ is $C^{\alpha}$ at $\Gamma$.

In \cite{CFMS} and \cite{FGMS}, the boundary Harnack inequality was extended to divergence and non-divergence equations, with $A$ just measurable. Other versions for equations with right-hand side like $Lu_{i}=f_{i}$ were studied more recently, see Allen and Shangholian \cite{otherrhs} and Ros-Oton and Torres-Latorre \cite{RHS}.

The Lipschitz condition of the boundary $\Gamma$ can also be weakened. For example, Jerison and Kenig in \cite{JK82} extended the result to NTA domains, and Bass and Burdzy in \cite{BB1}, \cite{BB2} to twisted H\"older domains by probabilistic methods. More precisely, $\Gamma$ is assumed to be $C^{\alpha}$ with $\alpha>1/2$ if $L$ is in non-divergence form, or with $\alpha>0$ if $L$ is a divergence operator. A unified analytic approach was given in \cite{DS19} by De Silva and Savin.

\subsection{Higher order boundary Harnack}
Next we assume $\Gamma\in C^{k,\alpha}$, $A\in C^{k-1,\alpha}$ for $k\geq1$. In \cite{DS14} De Silva and Savin showed the $C^{k,\alpha}$ regularity of the ratio $w$ by using higher order polynomials to approximate the $u_{i}$'s. They also gave the higher order boundary Harnack for a slit domain in \cite{DS14thin} using a similar method.

In this paper, we will prove the higher order boundary Harnack inequality by a different method. 
\begin{theorem}\label{HOBH}
    Assume that $u_{1},u_{2}>0$ defined on $\Omega_{1}=\{x_{n}>\Gamma(x')\}\cap B_{1}$ satisfy
    \begin{equation}
        div(A\cdot\nabla u_{i})=0,\quad u_{i}\Big|_{x_{n}=\Gamma(x')}=0
    \end{equation}
    with $A$ symmetric and uniformly elliptic. If $\Gamma\in C^{k,\alpha}$, $A\in C^{k-1,\alpha}$ for $k\geq1$, then the ratio $\ds w=\frac{u_{1}}{u_{2}}$ is $C^{k,\alpha}$ in $\Omega_{1/2}$.
\end{theorem}
The strategy to obtain Theorem \ref{HOBH} is to straighten the boundary $\Gamma$ by the coordinate change
\begin{equation}\label{ytox abuse}
    (x',x_{n})\to(x',x_{n}-\Gamma(x')),
\end{equation}
and investigate the degenerate equation satisfied by $w$. By a Schauder type estimate of $w$ (see Theorem \ref{degenerateschauder} below), we can prove that if $\Gamma\in C^{1,\alpha}$, then $w\in C^{1,\alpha}$. The general result can be deduced from the $C^{1,\alpha}$ estimate by taking tangential derivatives.

\subsection{A degenerate equation} In our setting $L=div(A\cdot\nabla)$ is a uniformly-elliptic divergence operator. If $A,\nabla_{x'}\Gamma\in C^{\alpha}$, then after the coordinate change \eqref{ytox abuse}, $u_{i}$ solves the equation $\ds div(A\cdot\nabla u_{i})=0$ with $A\in C^{\alpha}(B_{1}^{+})$. The ratio $\ds w=\frac{u_{1}}{u_{2}}$ satisfies
\begin{equation*}
    div(u_{2}^{2}A\cdot\nabla w)=0
\end{equation*}
as long as $A$ is symmetric. This equation can be turned into a model degenerate equation
\begin{equation}\label{main}
    div(x_{n}^{2}A\cdot\nabla w)=div(x_{n}^{2}\vec{f})\quad\mbox{in}\ B_{1}^{+}.
\end{equation}

The equation \eqref{main} belongs to the class of general $x_{n}^{s}$-weight equations
\begin{equation}\label{sweight}
    div(x_{n}^{s}A\cdot\nabla w)=div(x_{n}^{s}\vec{f}).
\end{equation}
When $s\in(-1,1)$, $x_{n}^{s}$ is an $A_{2}$-weight, which were studied extensively in the literature, see \cite{F82a},\cite{F82b}. There are also results outside this range. For example, in \cite{DP20} and \cite{DP21} Dong and Phan gave a weighted $W^{2,p}$ estimate for the Neumann problem when $s>-1$ and the Dirichlet problem when $s<1$.

In our setting $s=2$, and we prove the following Schauder estimate:
\begin{theorem}\label{degenerateschauder}
Assume $A$ is a $C^{\alpha}$ matrix with $\lambda I\leq A\leq\Lambda I$. If $w\in L^{2}(B_{1}^{+})\cap H^{1}_{loc}(B_{1}^{+})$ satisfies \eqref{main}, then there exists $C=C(n,\alpha,\lambda,\Lambda,\|A\|_{C^{\alpha}(B_{1}^{+})})$, so that
\begin{equation}
    \|w\|_{C^{1,\alpha}(B_{1/2}^{+})}\leq C(\|\vec{f}\|_{C^{\alpha}(B_{1}^{+})}+\|w\|_{L^{2}(B_{1}^{+})}).
\end{equation}
\end{theorem}
We prove Theorem \ref{degenerateschauder} through the standard Campanato iteration, together with the weighted Poincare and Caccioppoli inequality. The method can be easily generalized to equation \eqref{sweight} with $s>1$.

We remark that we cannot pose a Dirichlet boundary problem to \eqref{main} at the planar boundary $B_{1}'=B_{1}\cap\{x_{n}=0\}$. Instead, $w\in L^{2}(B_{1}^{+})$ can be viewed as a Neumann boundary data of \eqref{main} at $B_{1}'$. More generally, the Neumann boundary data of \eqref{sweight} takes the form $\ds\lim_{x_{n}\to0}x_{n}^{s-1}w=0$ in trace sense.

\subsection{Applications} We can use the higher order boundary Harnack to study the regularity in obstacle problems.

Assume that $\lambda I\leq A\leq\Lambda I$ is analytic. Let $U\geq0$ be a $C^{1,\alpha}_{loc}$ solution of the classical obstacle problem
\begin{equation}\label{U in y abuse}
    div(A\cdot\nabla U)=\chi_{\{U>0\}},
\end{equation}
 where $0<\alpha<1$ can be arbitrarily closed to $1$. If a sub-sequence of the blow up at $x_{0}\in\Gamma$
\begin{equation}
    U_{r,x_{0}}(x):=\frac{U(rx+x_{0})}{r^{2}}
\end{equation}
converges uniformly locally to a half-plane solution like $\ds\frac{k}{2}|(x\cdot e)^{+}|^{2}$ as $r\searrow0$, then $x_{0}$ is called a regular free boundary point of $U$. The set of all regular free boundary points is referred to as the regular free boundary.

In the seminal paper \cite{Caff77} Caffarelli first proved that the free boundary
\begin{equation*}
    \Gamma:=\partial\{U>0\}
\end{equation*}
is a $C^{0,1}$ graph near a regular free boundary point $X_{0}$, and all the free boundary points near $X_{0}$ are regular. If $\Gamma$ is Lipschitz, since $\ds\frac{U_{x_{i}}}{U_{x_{n}}}$ represents the slope of level sets of $U$, a boundary Harnack estimate of this ratio yields that $\Gamma\in C^{1,\alpha}$. The higher regularity and analyticity of $\Gamma$ was established by Kinderlehrer and Nirenberg in \cite{KN77} using a hodograph-Legendre transformation.

A second approach is to inductively apply higher order boundary Harnack to obtain higher regularity of $\Gamma$. The idea is quite simple: assume that $\Gamma\in C^{k,\alpha}$, then Theorem \ref{HOBH} applied to $\ds w=\frac{U_{x_{i}}}{U_{x_{n}}}$ implies that $\partial_{i}\Gamma\in C^{k,\alpha}$, hence $\Gamma\in C^{k+1,\alpha}$. For example, in \cite{DS14} (or \cite{DS14thin} for the thin obstacle problem) De Silva and Savin proved that the free boundary $\Gamma\in C^{\infty}$ in a small neighbourhood of a regular free boundary point.

In this paper, we first straighten the boundary $\Gamma$ using \eqref{ytox abuse}, and inductively apply Theorem \ref{degenerateschauder} to $\ds D^{k}_{T}w=D^{k}_{T}(\frac{U_{x_{i}}}{U_{x_{n}}})$. With more careful computations, we provide an alternative proof of the analyticity of the free boundary.
\begin{theorem}\label{analytic}
    Let $U$ be a solution of \eqref{U in y abuse} with $\lambda I\leq A\leq\Lambda I$ being analytic. If $0\in\Gamma$ is a regular free boundary point, then $\Gamma$ is an analytic curve near $0$.
\end{theorem}
We follow Blatt's method of \cite{blatt} and construct power series encoding the higher regularity of $u_{i}=U_{x_{i}}$ and $\ds w_{i}=\frac{U_{x_{i}}}{U_{x_{n}}}$. The elliptic estimates give inductive inequalities of the higher regularity, which could be turned into an ODE inequality system about the two power series.

\subsection{More comments} 
The study of weighted degenerate equations like \eqref{main} can be used to study other problems like the Alt-Phillips free boundary problems, or the thin obstacle problem. We intend to use the methods developed in this paper to study the analyticity in these problems. This would provide a different proof of the analyticity in the thin obstacle problem, which was established in \cite{KPS14} by Koch, Petrosyan and Shi.

While working on this note, the preprint \cite{TTV} by Terracini, Tortone and Vita appeared, in which similar Schauder estimates for weighted equations \eqref{sweight} when $s>-1$ are established.

\section{A weighted Sobolev space}
\subsection{Poincare inequality}Let's first establish a Poincare type inequality (Proposition \ref{regular poincare}) which fits the degenerate equation.
\begin{lemma}\label{L1poin}
If $w\in C^{1}_{loc}(\mathbb{R}^{n}_{+})$ is compactly supported, which means there exists $R>0$ such that $w\Big|_{|x|>R}=0$, then we have the inequality
\begin{equation}\label{L1 poincare}
    \int_{\mathbb{R}^{n}_{+}}|w|dx\leq\int_{\mathbb{R}^{n}_{+}}x_{n}|\partial_{n}w|dx.
\end{equation}
\end{lemma}
\begin{proof}
Fix $x'=(x_{1},...,x_{n-1})$, and it suffices to show that for each $x'$,
\begin{equation*}
    \int_{0}^{\infty}|w(x',x_{n})|dx_{n}\leq\int_{0}^{\infty}x_{n}|\partial_{n}w(x',x_{n})|dx_{n}.
\end{equation*}
To show this, we notice that since $w$ is compactly supported,
\begin{equation*}
    |w(x',x_{n})|\leq\int_{x_{n}}^{\infty}|\partial_{n}w(x',y_{n})|dy_{n},
\end{equation*}
and hence by changing the order of integration, we have
\begin{equation*}
    \int_{0}^{\infty}|w|\leq\int_{0}^{\infty}\int_{0}^{y_{n}}|\partial_{n}w(x',y_{n})|dx_{n}dy_{n}=\int_{0}^{\infty}x_{n}|\partial_{n}w(x',x_{n})|dx_{n}.
\end{equation*}
\end{proof}
This lemma is also true for $u\in H^{1}_{loc}(\mathbb{R}^{n}_{+})$ compactly supported. By replacing $w$ with $w^{2}$ in \eqref{L1 poincare}, we obtain the following Poincare inequality.

\begin{proposition}\label{regular poincare}
If $w\in H^{1}_{loc}(\mathbb{R}^{n}_{+})$ is compactly supported, i.e. $w\Big|_{|x|>R}=0$, then
\begin{equation}
\int_{\mathbb{R}^{n}_{+}}w^{2}dx\leq4\int_{\mathbb{R}^{n}_{+}}x_{n}^{2}|\nabla w|^{2}dx.
\end{equation}
\end{proposition}

It's thus quite natural to define a weighted Sobolev space:
\begin{definition}
    The space $H^{1}(B_{r}^{+},x_{n}^{2}dx)$ is the set of $H^{1}_{loc}$ functions $w$ such that
\begin{equation*}
    \int_{B_{r}^{+}}x_{n}^{2}|\nabla w|^{2}dx+\int_{B_{r}^{+}}w^{2}dx<\infty.
\end{equation*}
\end{definition}
The space $H^{1}(B_{r}^{+},x_{n}^{2}dx)$ is a Hilbert space by defining the square norm as
\begin{equation*}
    \|w\|_{H^{1}(B_{r}^{+},x_{n}^{2}dx)}^{2}:=\int_{B_{r}^{+}}x_{n}^{2}|\nabla w|^{2}dx+\int_{B_{r}^{+}}w^{2}dx.
\end{equation*}

For each $w\in H^{1}(B_{r}^{+},x_{n}^{2}dx)$, there exists $w_{i}\in H^{1}(B_{r}^{+},x_{n}^{2}dx)\cap C^{\infty}(\overline{B_{r}^{+}})$, so that
\begin{equation*}
    \int_{B_{r}^{+}}x_{n}^{2}|\nabla(w-w_{i})|^{2}dx+\int_{B_{r}^{+}}(w-w_{i})^{2}dx\to0.
\end{equation*}

\subsection{Caccioppoli inequalities}Now, let's also develop two Caccioppoli inequalities for the degenerate model \eqref{main}. First we establish an interior estimate.

\begin{proposition}\label{caccioppoli}
Assume that $\lambda I\leq A\leq\Lambda I$. If $w\in H^{1}_{loc}(B_{r}^{+})$ is a solution of \eqref{main}, then there exists $C=C(n,\lambda,\Lambda)$ so that
\begin{equation}
    \int_{B_{r/2}^{+}}x_{n}^{2}|\nabla w|^{2}dx\leq C\Big\{\int_{B_{r}^{+}}x_{n}^{2}|\vec{f}|^{2}dx+\int_{B_{r}^{+}}|w|^{2}dx\Big\}.
\end{equation}
\end{proposition}
\begin{proof}
Let $\varphi(x),\eta_{h}(x)$ be two positive smooth functions satisfying
\begin{align}
    &\varphi\Big|_{|x|\leq r/2}=1,\quad\varphi\Big|_{|x|\geq r}=0,\quad|\nabla\varphi|\leq\frac{C}{r},\\
    &\eta_{h}\Big|_{x_{n}\geq h}=1,\quad\eta_{h}\Big|_{x_{n}\leq0}=0,\quad|\nabla\eta_{h}|\leq\frac{C}{h}.
\end{align}
We also denote $\varphi_{h}=\varphi\cdot\eta_{h}$ and it's not hard to show that $x_{n}^{2}|\nabla\varphi_{h}|^{2}\leq C$ inside $B_{r}^{+}$. Now multiplying $\varphi_{h}^{2}w$ on both sides of \eqref{main}, integration by parts implies
\begin{equation*}
\int_{B_{r}^{+}}x_{n}^{2}A\Big(\nabla(\varphi_{h}w),\nabla(\varphi_{h}w)\Big)-\int_{B_{r}^{+}}x_{n}^{2}w^{2}A(\nabla\varphi_{h},\nabla\varphi_{h})=\int_{B_{r}^{+}}x_{n}^{2}\vec{f}\cdot\nabla(\varphi_{h}^{2}w).
\end{equation*}
As $\lambda I\leq A\leq\Lambda I$ and $x_{n}^{2}|\nabla\varphi_{h}|^{2}\leq C$, we have
\begin{align*}
\int_{B_{r}^{+}}x_{n}^{2}A\Big(\nabla(\varphi_{h}w),\nabla(\varphi_{h}w)\Big)\geq&\lambda\int_{B_{r}^{+}}x_{n}^{2}|\nabla(\varphi_{h}w)|^{2},\\
0\leq\int_{B_{r}^{+}}x_{n}^{2}w^{2}A(\nabla\varphi_{h},\nabla\varphi_{h})\leq&C\cdot\Lambda\int_{B_{r}^{+}}w^{2}.
\end{align*}
Besides, as $0\leq\varphi_{h}\leq1$,
\begin{equation*}
|\int_{B_{r}^{+}}x_{n}^{2}\vec{f}\cdot\nabla(\varphi_{h}^{2}w)|\leq\frac{\varepsilon}{2}\int_{B_{r}^{+}}x_{n}^{2}|\nabla(\varphi_{h}w)|^{2}+C\varepsilon\int_{B_{r}^{+}}w^{2}+\frac{1}{\varepsilon}\int_{B_{r}^{+}}x_{n}^{2}|\vec{f}|^{2}.
\end{equation*}
Putting those estimates together yields that when $\varepsilon=\varepsilon(n,\lambda,\Lambda)$,
\begin{align*}
\int_{B_{r/2}^{+}\cap\{x_{n}\geq h\}}x_{n}^{2}|\nabla w|^{2}\leq\int_{B_{r}^{+}}x_{n}^{2}|\nabla(\varphi_{h}w)|^{2}\leq C\Big\{\int_{B_{r}^{+}}w^{2}+\int_{B_{r}^{+}}x_{n}^{2}|\vec{f}|^{2}\Big\}.
\end{align*}
Finally, we just need to apply Fatou lemma to $\ds x_{n}^{2}|\nabla w|^{2}=\lim_{h\to0}x_{n}^{2}|\nabla w|^{2}\chi_{\{x_{n}\geq h\}}$.
\end{proof}
\begin{remark}
    A more careful computation yields $\ds x_{n}^{2}|\nabla\varphi_{h}|^{2}\leq C\Big(\frac{x_{n}^{2}}{r^{2}}+\chi_{\{x_{n}\leq h\}}\Big)$, so if $\ds\int_{B_{r}^{+}}w^{2}dx<\infty$, we can get a more accurate Caccioppoli inequality:
    \begin{equation}
    \int_{B_{r/2}^{+}}x_{n}^{2}|\nabla w|^{2}dx\leq C\Big\{\int_{B_{r}^{+}}x_{n}^{2}|\vec{f}|^{2}dx+\frac{1}{r^{2}}\int_{B_{r}^{+}}x_{n}^{2}|w|^{2}dx\Big\}.
\end{equation}
\end{remark}

Next, here is a global inequality for solutions of \eqref{main} with zero boundary data.

\begin{proposition}\label{dirichlet}
Assume that $\lambda I\leq A\leq\Lambda I$. If $w\in L^{2}(B_{r}^{+})\cap H^{1}_{loc}(B_{r}^{+})$ is a solution of \eqref{main} so that $w\Big|_{|x|\geq r}=0$, then there exists $C=C(n,\lambda,\Lambda)$ so that
\begin{equation}
    \int_{B_{r}^{+}}x_{n}^{2}|\nabla w|^{2}dx\leq C\int_{B_{r}^{+}}x_{n}^{2}|\vec{f}|^{2}dx.
\end{equation}
\end{proposition}

\begin{proof}
The proof is similar to Proposition \ref{caccioppoli} and we adopt its cut-off function $\eta_{h}$, which satisfies $x_{n}^{2}|\nabla\eta_{h}|^{2}\leq C\chi_{\{x_{n}\leq h\}}$. We multiply by $\eta_{h}^{2}w$ on both sides of \eqref{main} and do integration by parts. Finally, we let $h\to0$ and apply the Fatou lemma.
\end{proof}

\subsection{Weak solution}
Assume that $\lambda I\leq A\leq\Lambda I$. To find a solution of \eqref{main} satisfying some boundary condition, we can use the Lax-Milgram lemma. Let the bilinear form and linear functional
\begin{equation}
B(w,\varphi):=\int_{B_{r}^{+}}x_{n}^{2}\nabla\varphi^{T}A\nabla w dx,\quad F(\varphi)=\int_{B_{r}^{+}}x_{n}^{2}\vec{f}\cdot\nabla\varphi dx
\end{equation}
be defined on $H^{1}(B_{r}^{+},x_{n}^{2}dx)$. If $w=0$ when $|x|>r$, then by Proposition \ref{regular poincare},
\begin{equation*}
    B(w,w)\geq\frac{\lambda}{5}\|w\|^{2}_{H^{1}(B_{r}^{+},x_{n}^{2}dx)},\quad B(u,v)\leq\Lambda\|u\|_{H^{1}(B_{r}^{+},x_{n}^{2}dx)}\|v\|_{H^{1}(B_{r}^{+},x_{n}^{2}dx)}.
\end{equation*}

In the special case $\vec{f}=0$, we have the following existence result:
\begin{proposition}\label{exist}
    Assume that $\lambda I\leq A\leq\Lambda I$. If $w_{0}\in H^{1}(B_{r}^{+},x_{n}^{2}dx)$, then there exists a unique weak solution $w\in H^{1}(B_{r}^{+},x_{n}^{2}dx)$ of
    \begin{equation*}
        div(x_{n}^{2}A\nabla w)=0
    \end{equation*}
    with boundary data $w_{0}$. Besides, there exists $C=C(n,\lambda,\Lambda)$ so that
    \begin{equation}
        \int_{B_{r}^{+}}x_{n}^{2}|\nabla w|^{2}dx\leq C\int_{B_{r}^{+}}x_{n}^{2}|\nabla w_{0}|^{2}dx.
    \end{equation}
\end{proposition}
\begin{proof}
    It suffices to find a solution $w_{1}\in H^{1}(B_{r}^{+},x_{n}^{2}dx)$ so that $\ds w_{1}\Big|_{|x|>r}=0$ and
    \begin{equation*}
        div(x_{n}^{2}A\nabla w_{1})=-div(x_{n}^{2}A\nabla w_{0}).
    \end{equation*}
    Let $\vec{f}=-A\nabla w_{0}$, then as $w_{0}\in H^{1}(B_{r}^{+},x_{n}^{2}dx)$, the linear functional $F(\varphi)$ is bounded. By Lax-Milgram lemma, $w_{1}$ exists and
    \begin{equation*}
        \int_{B_{r}^{+}}x_{n}^{2}|\nabla w_{1}|^{2}\leq C\int_{B_{r}^{+}}x_{n}^{2}|\vec{f}|^{2}\leq C\int_{B_{r}^{+}}x_{n}^{2}|\nabla w_{0}|^{2}.
    \end{equation*}
    If $w=w_{0}+w_{1}$, then $div(x_{n}^{2}A\nabla w)=0$ with boundary data $w_{0}$.
\end{proof}
\begin{remark}
    We should notice that $w_{0}$ is the Dirichlet boundary data of $w$ only on the spherical part of the hemisphere. Even if 
$\|w_{0}\|_{L^{\infty}}<\infty$, what we know on the planar boundary $B_{r}'$ is at most $|w|<\infty$. Therefore, we may call $w_{0}$ a half-Dirichlet-half-Neumann boundary data.
\end{remark}

\section{Schauder estimate}
In this section, we will give a $C^{1,\alpha}$ interior estimate of model equation \eqref{main}. First, Proposition \ref{gammaschauder} gives a Schauder estimate at the planar boundary $B_{1/2}'$, then Theorem \ref{degenerateschauder} follows easily by interior Schauder estimate.

We first give two approximation lemmas.
\begin{lemma}\label{constant coefficient}
If $h\in H^{1}(B_{1}^{+},x_{n}^{2}dx)$ solves the equation
\begin{equation*}
    div(x_{n}^{2}\nabla h)=0,
\end{equation*}
then $h\in C^{2}(B_{1/2}^{+})$. In particular, there exists a linear polynomial
\begin{equation}
    l(x)=h(0)+x\cdot\nabla h(0)
\end{equation}
such that for all $\rho\leq1/2$,
\begin{equation}
    \|l\|_{C^{0,1}(B_{1}^{+})}^{2}+\frac{1}{\rho^{n+4}}\int_{B_{\rho}^{+}}|h-l|^{2}dx\leq C(n)\|h\|_{L^{2}(B_{1}^{+})}^{2}.
\end{equation}
\end{lemma}
\begin{proof}
Let $h_{k}\in H^{1}(B_{1}^{+},x_{n}^{2}dx)\cap C^{\infty}(\overline{B_{1}^{+}})$ be a $H^{1}(B_{1}^{+},x_{n}^{2}dx)$ approximation of $h$, then $H_{k}=x_{n}h_{k}$ is a smooth $H^{1}(B_{1}^{+},dx)$ approximation of $H=x_{n}h$ and each $H_{k}$ vanishes at $\{x_{n}=0\}$. Therefore, $H\Big|_{B_{1}'}=0$ in trace sense. Notice that $\Delta H=0$, and the desired inequality follows from the $C^{3}$ boundary estimate of $H$.
\end{proof}
\begin{remark}
    One can also show that $l(x)=l(x')$ is vertically constant. In fact, for linear functions, $l(x)=l(x')$ is equivalent to $div(x_{n}^{2}\nabla l)=0$.
\end{remark}
\begin{remark}
    You can refer to \cite{DP20} for a more general proof of Lemma \ref{constant coefficient}. The method is to establish a De Giorgi - Nash - Moser estimate in the weighted setting.
\end{remark}

Later, we will use the following rescaled version of the previous lemma.
\begin{corollary}\label{constant coefficient corollary}
    Let $r>0$. If $h\in H^{1}(B_{r}^{+},x_{n}^{2}dx)$ solves the equation
\begin{equation*}
    div(x_{n}^{2}\nabla h)=0,
\end{equation*}
then $h\in C^{2}(B_{r/2}^{+})$. In particular, there exists a linear polynomial
\begin{equation}
    l(x)=h(0)+x\cdot\nabla h(0)
\end{equation}
such that for all $\rho\leq1/2$,
\begin{equation}
    |l(0)|^{2}+r^{2}|\nabla l|^{2}+\frac{1}{\rho^{n+4}r^{n}}\int_{B_{\rho r}^{+}}|h-l|^{2}dx\leq\frac{C(n)}{r^{n}}\|h\|_{L^{2}(B_{r}^{+})}^{2}.
\end{equation}
\end{corollary}

\begin{lemma}\label{harmonic replacement}
Assume that $a^{ij}(0)=\delta^{ij}$ and $[a^{ij}]_{C^{\alpha}(B_{r}^{+})}\leq\varepsilon_{0}$. Let $w\in H^{1}_{loc}(B_{r}^{+})$ be a solution of \eqref{main}, then there is a weak solution
\begin{equation}
    div\Big(x_{n}^{2}\nabla h\Big)=0,\quad h\Big|_{(\partial B_{r/2})^{+}}=w
\end{equation}
in the space $H^{1}(B_{r/2}^{+},x_{n}^{2}dx)$, and there exists $C=C(n)$ so that
\begin{align}
    \int_{B_{r/2}^{+}}h^{2}dx\leq&C\Big\{\int_{B_{r}^{+}}x_{n}^{2}|\vec{f}|^{2}dx+\int_{B_{r}^{+}}w^{2}dx\Big\},\label{h1}\\
    \int_{B_{r/2}^{+}}|h-w|^{2}dx\leq&C\Big\{\int_{B_{r}^{+}}x_{n}^{2}|\vec{f}|^{2}dx+\varepsilon_{0}^{2}r^{2\alpha}\int_{B_{r}^{+}}w^{2}dx\Big\}.\label{h2}
\end{align}
\end{lemma}
\begin{proof}
We first use Proposition \ref{caccioppoli} and \ref{exist} to get the existence of $h$ with
\begin{equation*}
    \int_{B_{r/2}^{+}}x_{n}^{2}|\nabla h|^{2}\leq C\int_{B_{r/2}^{+}}x_{n}^{2}|\nabla w|^{2}\leq C\Big\{\int_{B_{r}^{+}}x_{n}^{2}|\vec{f}|^{2}+\int_{B_{r}^{+}}w^{2}\Big\}.
\end{equation*}
If we denote $\ds\vec{f}'=\Big(a^{ij}(x)-\delta^{ij}\Big)\partial_{j}h\cdot e_{i}$, then
\begin{equation}\label{fprime}
    \int_{B_{r/2}^{+}}x_{n}^{2}|\vec{f}'|^{2}\leq C\varepsilon_{0}^{2}r^{2\alpha}\int_{B_{r/2}^{+}}x_{n}^{2}|\nabla h|^{2}\leq C\varepsilon_{0}^{2}r^{2\alpha}\Big\{\int_{B_{r}^{+}}x_{n}^{2}|\vec{f}|^{2}+\int_{B_{r}^{+}}w^{2}\Big\}
\end{equation}
and $(h-w)$ satisfies $div\Big(x_{n}^{2}A(x)\cdot\nabla (h-w)\Big)=div\Big(x_{n}^{2}(\vec{f}'-\vec{f})\Big)$ with zero boundary data at $(\partial B_{r/2})^{+}$. Proposition \ref{regular poincare} and \ref{dirichlet}, applied to $(h-w)$, give
\begin{equation*}
    \int_{B_{r/2}^{+}}|h-w|^{2}\leq C\int_{B_{r/2}^{+}}x_{n}^{2}|\nabla(h-w)|^{2}\leq C\int_{B_{r/2}^{+}}x_{n}^{2}\big(|\vec{f}|^{2}+|\vec{f}'|^{2}\big).
\end{equation*}
By using \eqref{fprime} we prove \eqref{h2}. Now \eqref{h1} follows from triangle inequality.
\end{proof}

Now let's give a Schauder estimate of $w$ at $\mathbb{R}^{n-1}$.

\begin{proposition}\label{gammaschauder}
Assume that $a^{ij}(0)=\delta^{ij}$, $[a^{ij}]_{C^{\alpha}(B_{1}^{+})}\leq\varepsilon_{0}$ and $w\in L^{2}(B_{1}^{+})\cap H^{1}_{loc}(B_{1}^{+})$ is a solution of \eqref{main}. Given that $\varepsilon_{0}$ is small, there exist a linear polynomial $P$ and $C=C(n,\alpha)$ such that
\begin{equation}
    \|P\|_{C^{0,1}(B_{1}^{+})}^{2}+\frac{1}{r^{n+2+2\alpha}}\int_{B_{r}^{+}}|w-P|^{2}dx\leq C(\|\vec{f}\|_{C^{\alpha}(B_{1})}^{2}+\|w\|_{L^{2}(B_{1})}^{2}).
\end{equation}
\end{proposition}
\begin{proof}
Let's first deal with the case $\vec{f}(0)=0$. Let's inductively define
\begin{equation}
    w_{0}=w,\quad w_{k+1}=w_{k}-l_{k}(k\geq0),\quad P_{k}=w-w_{k}=\sum_{i=0}^{k}l_{i},
\end{equation}
where the linear polynomials $l_{k}(x)=l_{k}(x')$ will be chosen in \eqref{latereplace}. We also define
\begin{equation}
\vec{f}_{0}=\vec{f},\quad\vec{f}_{k+1}=\vec{f}_{k}+\Big(\delta^{ij}-a^{ij}(x)\Big)\Big(\partial_{j}l_{k}\Big)e_{i},(k\geq0).
\end{equation}
By induction, we know $\vec{f}_{k}(0)=0$ and $div(x_{n}^{2}A\cdot\nabla w_{k})=div(x_{n}^{2}\vec{f}_{k})$ for all $k\geq0$. Let $\ds\varepsilon_{0}^{\frac{2}{n+4}}\leq S\leq\frac{1}{4}$ be a shrinking rate, and we define two quantities
\begin{equation}
    \sigma_{k}^{2}:=\frac{1}{S^{k(n+2+2\alpha)}}\int_{B_{S^{k}}^{+}}w_{k}^{2}dx,\quad\chi_{k}:=[\vec{f}_{k}]_{C^{\alpha}(B_{S^{k}}^{+})}.
\end{equation}
Here, $\sigma_{k}$ measures how $w$ and $P_{k}$ differ in a $C^{1,\alpha}$ sense. Let $h_{k}\in H^{1}(B_{S^{k}/2}^{+},x_{n}^{2}dx)$ be a replacement of $w_{k}$ in $B_{S^{k}/2}^{+}$, and $l_{k}$ be the linearization of $h_{k}$, i.e.
\begin{equation}\label{latereplace}
    \left\{
    \begin{aligned}
    &div\Big(x_{n}^{2}\nabla h_{k}\Big)=0\\
    &h\Big|_{(\partial B_{S^{k}/2})^{+}}=w_{k}
    \end{aligned}
    \right.,\quad l_{k}=h_{k}(0)+x\cdot\nabla h_{k}(0).
\end{equation}
By applying Lemma \ref{harmonic replacement} and Corollary \ref{constant coefficient corollary} to $w_{k+1}=(w_{k}-h_{k})+(h_{k}-l_{k})$, we have
\begin{equation}\label{iteration inequality}
\left\{
    \begin{aligned}
    &\chi_{k+1}\leq C(n)(\chi_{k}+\varepsilon_{0}S^{k\alpha}\sigma_{k}),\\
    &\sigma_{k+1}^{2}\leq C(n)S^{2-2\alpha}\sigma_{k}^{2}+\frac{C(n)}{S^{n+2+2\alpha}}\chi_{k}^{2},\\
    &|l_{k}(0)|+|\nabla l_{k}(0)|\leq C(n)S^{k\alpha}(\chi_{k}+\sigma_{k}).
    \end{aligned}
    \right.
\end{equation}
Clearly, when $\varepsilon_{0}$ is very small, then there is room to choose a $S=S(n,\alpha)$ satisfying
\begin{equation*}
    \varepsilon_{0}^{\frac{2}{n+4}}\leq S\leq\frac{1}{4},\quad C(n)S^{2-2\alpha}\leq\frac{1}{2}.
\end{equation*}
With this the iteration inequality \eqref{iteration inequality} implies
\begin{equation*}
    \sigma_{k},\chi_{k}\leq C(n,\alpha)(\chi_{0}+\sigma_{0})=C(\|\vec{f}\|_{C^{\alpha}(B_{1})}+\|w\|_{L^{2}(B_{1})}).
\end{equation*}
Besides, it follows that the sequence of linear polynomials $P_{k}$ converges to $P$ with
\begin{equation*}
    \|P\|_{C^{0,1}(B_{1}^{+})}\leq C(\|\vec{f}\|_{C^{\alpha}(B_{1})}+\|w\|_{L^{2}(B_{1})}).
\end{equation*}
$P$ is a $C^{1,\alpha}$ approximation of $w$ because $\sigma_{k}$'s are bounded.

Finally, if $\vec{f}(0)\neq0$, then we can subtract a linear function from $w$ so that the remainder $w'$ satisfies $div(x_{n}^{2}A\cdot\nabla w')=div(x_{n}^{2}\vec{f}')$ with $\vec{f}'(0)=0$. Now we are reduced to the first case.
\end{proof}
Now, Theorem \ref{degenerateschauder} is almost obvious by an interior estimate.

{\it Proof of Theorem \ref{degenerateschauder}}.
After rescaling, we can assume that $a^{ij}(0)=\delta^{ij}$ $[a^{ij}]_{C^{\alpha}}\leq\varepsilon_{0}$. Let $P(x)$ be the linear approximation of $w$, then it suffices to prove the regularity of $(w-P)$, which satisfies
\begin{equation}\label{error equation}
    div\Big(A\cdot\nabla(w-P)\Big)=div(x_{n}^{2}\vec{f}_{\infty})
\end{equation}
with $\ds\vec{f}_{\infty}=\lim_{k\rightarrow\infty}\vec{f}_{k}$. It's not hard to check that
\begin{equation*}
    \vec{f}_{\infty}(0)=0,\quad[\vec{f}_{\infty}]_{C^{\alpha}(B_{1}^{+})}\leq C(\|\vec{f}\|_{C^{\alpha}(B_{1}^{+})}+\|w\|_{L^{2}(B_{1}^{+})}).
\end{equation*}
Besides, Proposition \ref{gammaschauder} gives a $L^{2}$ bound of the error $(w-P)$:
\begin{equation*}
    \frac{1}{r^{n+2+2\alpha}}\int_{B_{r}^{+}}|w-P|^{2}\leq C(\|\vec{f}\|_{C^{\alpha}(B_{1}^{+})}^{2}+\|w\|_{L^{2}(B_{1}^{+})}^{2}).
\end{equation*}
This holds not only when $B_{r}^{+}$ is centered at $0$, but its center can also move on $\mathbb{R}^{n-1}$. Therefore, we can apply the uniformly-elliptic Schauder estimate to \eqref{error equation} in small interior balls to get $\ds\|\nabla(w-P)\|_{C^{\alpha}(B_{1/2})}\leq C(\|\vec{f}\|_{C^{\alpha}(B_{1}^{+})}+\|w\|_{L^{2}(B_{1}^{+})})$.
\qed

\begin{remark}
    The right hand side of \eqref{main} can also have a non-divergence term, like
\begin{equation}\label{degenerate}
    div(x_{n}^{2}A\cdot\nabla w)=div(x_{n}^{2}\vec{f})+x_{n}g.
\end{equation}
If $g\in C^{\alpha}$, then $x_{n}g$ is absorbed by $div(x_{n}^{2}\vec{f})$, so Theorem \ref{degenerateschauder} immediately implies
\begin{equation}
    \|w\|_{C^{1,\alpha}(B_{1/2}^{+})}\leq C(\|\vec{f}\|_{C^{\alpha}(B_{1}^{+})}+\|g\|_{C^{\alpha}(B_{1}^{+})}+\|w\|_{L^{2}(B_{1}^{+})}).
\end{equation}
\end{remark}
\begin{remark}
    In Theorem \ref{degenerateschauder}, the assumption that $A$ is symmetric can be removed. In fact, Lemma \ref{constant coefficient} also works for non-symmetric matrix.
\end{remark}

\section{Higher order boundary Harnack}\label{relation}

\subsection{Straight boundary case}
We first prove Theorem \ref{HOBH} when $\Gamma=B_{1}'$ is straight. Before that, let's look at the following ratio lemma.
\begin{lemma}\label{ratio}
    If $u_{1},u_{2}$ satisfy $div(A\cdot\nabla u_{i})=div(\vec{f}_{i})$, then the ratio $\ds w=\frac{u_{1}}{u_{2}}$ satisfies
    \begin{equation}
        div(u_{2}^{2}A\nabla w)=div(u_{2}\vec{f}_{1}-u_{1}\vec{f_{2}})+(\vec{f}_{2}\cdot\nabla u_{1}-\vec{f}_{1}\cdot\nabla u_{2})
    \end{equation}
    as long as $A$ is symmetric.
\end{lemma}
\begin{proof}
    As $u_{1}=u_{2}w$, by expanding $\ds u_{2}div\Big(A\nabla(u_{2}w)\Big)=u_{2}div(\vec{f}_{1})$, we have
    \begin{equation*}
        u_{2}^{2}div(A\nabla w)+2u_{2}A(\nabla u_{2},\nabla w)=u_{2}div(\vec{f}_{1})-u_{2}w\cdot div(A\nabla u_{2}).
    \end{equation*}
    Its left-hand side is
    \begin{equation*}
        u_{2}^{2}div(A\nabla w)+A(\nabla u_{2}^{2},\nabla w)=div(u_{2}^{2}A\nabla w),
    \end{equation*}
    and the right-hand side is
    \begin{equation*}
        u_{2}div(\vec{f}_{1})-u_{1}div(\vec{f}_{2})=div(u_{2}\vec{f}_{1}-u_{1}\vec{f_{2}})+(\vec{f}_{2}\cdot\nabla u_{1}-\vec{f}_{1}\cdot\nabla u_{2}).
    \end{equation*}
\end{proof}
Our straight boundary version of Theorem \ref{HOBH} is the following, where we allow some right-hand side.
\begin{theorem}\label{HOBH straight}
    Assume that $\lambda I\leq A\leq\Lambda I$ and $u_{1},u_{2}>0$ defined on $B_{1}$ satisfy
    \begin{equation}\label{standardx}
        div(A\cdot\nabla u_{i})=div(x_{n}\vec{f}_{i}),\quad u_{i}\Big|_{B_{1}'}=0.
    \end{equation}
    \begin{itemize}
        \item[(a)] If $D^{k-1}_{T}\vec{f}_{i},D^{k-1}_{T}A\in C^{\alpha}$ for some $k\geq1$, then the ratio $\ds w=\frac{u_{1}}{u_{2}}$ is tangentially $C^{k,\alpha}$. More precisely, $D^{k-1}_{T}w\in C^{1,\alpha}(B_{1/2})$.
        \item[(b)] If $\vec{f}_{i},A\in C^{k-1,\alpha}$ for some $k\geq1$, then the ratio $\ds w=\frac{u_{1}}{u_{2}}$ is $C^{k,\alpha}(B_{1/2})$.
    \end{itemize}
\end{theorem}
Before proving Theorem \ref{HOBH straight}, let's recall that for a non-degenerate equation
\begin{equation}\label{standard}
    \partial_{i}(a^{ij}\partial_{j}u)=div(\vec{f}),\quad u\Big|_{B_{1}'}=0,
\end{equation}
there is a standard $C^{1,\alpha}$ boundary estimate.
\begin{theorem}\label{standardschauder}
Assume that $A\in C^{\alpha}$ and $\lambda|\xi|^{2}\leq\xi^{T}A\xi$. If $u$ be a solution of \eqref{standard}, then there is $C=C(n,\alpha,\lambda,\|A\|_{C^{\alpha}(B_{1}^{+})})$ so that
\begin{equation}
    \|u\|_{C^{1,\alpha}(B_{1/2}^{+})}\leq C([\vec{f}]_{C^{\alpha}(B_{1}^{+})}+\|u\|_{L^{2}(B_{1}^{+})}).
\end{equation}
\end{theorem}
Notice that when $u$ vanishes at $B_{1}'$, then $u\in C^{1,\alpha}$ implies $\ds\frac{u}{x_{n}}\in C^{\alpha}$. With this let's prove Theorem \ref{HOBH straight}.

{\it Proof of Theorem \ref{HOBH straight}}. We first show part (a). When $k=1$, then first we know $\ds\nabla u_{i},\frac{u_{i}}{x_{n}}\in C^{\alpha}(B_{3/4}^{+})$ by Theorem \ref{standardschauder} and $\ds\frac{u_{i}}{x_{n}}\geq\varepsilon>0$ by Hopf lemma.
    By Lemma \ref{ratio}, we know $\ds w=\frac{u_{1}}{u_{2}}$ satisfies
    \begin{equation*}
        div(u_{2}^{2}A\nabla w)=div(x_{n}u_{2}\vec{f}_{1}-x_{n}u_{1}\vec{f_{2}})+x_{n}(\vec{f}_{2}\cdot\nabla u_{1}-\vec{f}_{1}\cdot\nabla u_{2}).
    \end{equation*}
    By writing
    \begin{equation*}
        \tilde{A}=\frac{u_{2}^{2}}{x_{n}^{2}}A,\quad\vec{f}=\frac{u_{2}}{x_{n}}\vec{f}_{1}-\frac{u_{1}}{x_{n}}\vec{f_{2}},\quad g=\vec{f}_{2}\cdot\nabla u_{1}-\vec{f}_{1}\cdot\nabla u_{2},
    \end{equation*}
    we simplify the equation of $w$ to the form \eqref{degenerate}. The fact $\ds\frac{u_{2}}{x_{n}}\geq\varepsilon>0$ ensures that $\tilde{A}$ is still a uniformly elliptic matrix. Besides, $\tilde{A},\vec{f},g\in C^{\alpha}(B_{3/4}^{+})$. Therefore, $w\in C^{1,\alpha}(B_{1/2}^{+})$ by applying Theorem \ref{degenerateschauder}.
    
    When $k\geq2$, we first apply $D^{k-1}_{T}$ to \eqref{standardx} and obtain $\ds D^{k-1}_{T}u_{i}\in C^{1,\alpha}(B_{3/4}^{+})$. Therefore, $D^{k-1}_{T}\tilde{A},D^{k-1}_{T}\vec{f},D^{k-1}_{T}g\in C^{\alpha}(B_{3/4}^{+})$. By induction hypothesis, we can also assume $D^{k-2}w\in C^{1,\alpha}(B_{3/4}^{+})$. Now we just apply $D^{k-1}_{T}$ to \eqref{degenerate} and obtain $D^{k-1}w\in C^{1,\alpha}(B_{1/2}^{+})$.

    For part (b), we can inductively prove that for all $0\leq l\leq k$,
    \begin{equation*}
        D_{T}^{k-l}\partial_{n}^{l}w\in C^{\alpha}(B_{1/2}^{+}).
    \end{equation*}
    The base case $k=0,1$ is shown in part (a). For larger $l$, we just need to inductively take normal derivatives to \eqref{degenerate}.
\qed

\subsection{Boundary straightening}To avoid abuse of coordinate, from now on it's necessary to introduce two variables $x$ and $y$ to represent different coordinate systems. We call the original coordinate $y$-coordinate, and the new coordinate $x$-coordinate. That is, the coordinate change \eqref{ytox abuse}, is better written as
\begin{equation}\label{ytox}
    x'=y',\quad x_{n}=y_{n}-\Gamma(y').
\end{equation}
We can also treat $\Gamma$ as a function in $x$-coordinate, i.e. $\Gamma(x)=\Gamma(x')=\Gamma(y')$. The inverse coordinate change will be
\begin{equation}\label{xtoy}
    y(x)=x+\Gamma(x)e_{n}.
\end{equation}
If $u$ satisfies an equality in $y$-coordinate, then the following coordinate-change lemma gives the corresponding equality of $u$ in $x$-coordinate:
\begin{lemma}\label{straightening change}
Let $x=(x_{1},...,x_{n})$ and $y=(y_{1},...,y_{n})$ be two coordinate charts. If $\partial_{y_{p}}(b^{pq}\partial_{y_{q}}u)=\partial_{y_{p}}(f_{p})$, then
\begin{equation}
    \partial_{x_{i}}(a^{ij}\partial_{x_{j}}u)=\partial_{x_{i}}\Big(\det(\frac{\partial y}{\partial x})\frac{\partial x_{i}}{\partial y_{p}}f_{p}\Big),\ with\ a^{ij}=\det(\frac{\partial y}{\partial x})b^{pq}\frac{\partial x_{i}}{\partial y_{p}}\frac{\partial x_{j}}{\partial y_{q}}.
\end{equation}
\end{lemma}
\begin{proof}
Let $\varphi\in C^{\infty}_{0}$ be a test function, integration by parts in $y$-coordinate gives
\begin{equation*}
\int b^{pq}\partial_{y_{p}}\varphi\partial_{y_{q}}u dy=\int f_{p}\partial_{y_{p}}\varphi dy.
\end{equation*}
We turn both sides into integration in $x$-coordinate, then
\begin{equation*}
LHS=\int b^{pq}\Big(\frac{\partial x_{i}}{\partial y_{p}}\partial_{x_{i}}\varphi\Big)\Big(\frac{\partial x_{j}}{\partial y_{q}}\partial_{x_{j}}u\Big)\Big(\det(\frac{\partial y}{\partial x})dx\Big)=\int a^{ij}\partial_{x_{i}}\varphi\partial_{x_{j}}u dx,
\end{equation*}
and similarly
\begin{equation*}
RHS=\int f_{p}\Big(\frac{\partial x_{i}}{\partial y_{p}}\partial_{x_{i}}\varphi\Big)\det(\frac{\partial y}{\partial x})dx.
\end{equation*}
The desired identity now follows from integration by parts using $x$-coordinate.
\end{proof}
Now Theorem \ref{HOBH} is easily proven by combining Theorem \ref{HOBH straight} and Lemma \ref{straightening change}:

{\it Proof of Theorem \ref{HOBH}}. If $\Gamma\in C^{k,\alpha}$ then $\ds\frac{\partial x_{i}}{\partial y_{p}}$ is $C^{k-1,\alpha}$ in $y$-coordinate. By assumption, $b^{pq}\in C^{k-1,\alpha}$ in $y$-coordinate, and thus so does $a^{ij}$. Besides, $\Gamma\in C^{k,\alpha}$ also implies that the $x-y$ coordinate change is $C^{k,\alpha}$. More precisely, $y(x)$ is $C^{k,\alpha}$ in $x$-coordinate, so $a^{ij}(y(x))$ is also $C^{k-1,\alpha}$ in $x$-coordinate by chain rule. By applying Theorem \ref{HOBH straight} to the $x$-coordinate equation obtained by Lemma \ref{straightening change}, we arrive at the desired higher order boundary Harnack inequality.
\qed

\section{Growth rate at boundary and majorant power series}
\subsection{A global norm}
From now on, our goal will be to prove Theorem \ref{analytic}. To prove analyticity, we first develop global versions of Theorem \ref{degenerateschauder} and \ref{standardschauder} by controlling growth rate of the H\"older regularity near the spherical boundary $(\partial B_{1})^{+}$. For $X\in\mathbb{R}^{n}_{+}$, let's denote 
\begin{equation*}
    B_{r}^{+}(X):=B_{r}(X)\cap\mathbb{R}^{n}_{+}.
\end{equation*}

\begin{definition}
For any function $f$ defined in $B_{1}^{+}$ in $x$-coordinate, its global $C^{k,\alpha}$ norm with $b$-many normal derivatives is defined as
\begin{equation}
    [f]^{*,l}_{C_{b}^{k,\alpha}}:=\sup_{\substack{X\in B_{1}^{+}\\\Delta=1-|X|}}\max_{\substack{|\beta|=k\\\beta_{n}\leq b}}\Delta^{l}([D^{\beta}f]_{L^{\infty}(B_{\frac{\Delta}{l+1}}^{+}(X))}+\Delta^{\alpha}[D^{\beta}f]_{C^{\alpha}(B_{\frac{\Delta}{l+1}}^{+}(X))}).
\end{equation}
\end{definition}
In this paper, $b=0,1$, meaning that all (but at most one) derivatives are tangential. In our case we will choose $l=k$, see \eqref{power series} later.

This norm has a ``scaling invariance". If it's known that $[f]^{*,l}_{C_{b}^{k,\alpha}}\leq1$, then for any $X\in B_{1}^{+}$, we let $|z|<1$ such that $(X+\Delta_{X}\cdot z)\in B_{1}^{+}$, where $\Delta_{X}=1-|X|$, and set
    \begin{equation}
        f_{X,k,l}(z):=\Delta_{X}^{l-k}f(X+\Delta_{X}\cdot z),
    \end{equation}
    then we also have $[f_{X,k,l}]^{*,l}_{C_{b}^{k,\alpha}}\leq1$.

\subsection{Global degenerate Schauder estimate}

By using Theorem \ref{degenerateschauder} at $\mathbb{R}^{n-1}$ and using the non-degenerate Schauder estimates away from $\mathbb{R}^{n-1}$, Theorem \ref{degenerateschauder} holds for all $B_{1}^{+}(X)$. More precisely, we have the following lemma.
\begin{lemma}\label{centerhigher}
Let $X\in\mathbb{R}^{n}_{+}$ and $w\in L^{2}(B_{1}^{+}(X))$ satisfies \eqref{main} in $B_{1}^{+}(X)$. If $A\in C^{\alpha}$ with $\lambda I\leq A\leq\Lambda I$, then there is $C$ independent of the center $X$, so that
\begin{equation}
    \|w\|_{C^{1,\alpha}(B_{1/2}^{+}(X))}\leq C(\|\vec{f}\|_{C^{\alpha}(B_{1}^{+}(X))}+\|w\|_{L^{2}(B_{1}^{+}(X))}).
\end{equation}
\end{lemma}

Now let's give a global version of Theorem \ref{degenerateschauder}. In this proposition, we describe the boundary growth rate of the $C^{1,\alpha}$ estimate, and the shrinking rate of interior balls is generalized to $\ds\frac{1}{l+2}$.

\begin{proposition}\label{globalschauder}
Let $l\geq0$ be an arbitrary integer. Assume that $A\in C^{\alpha}$ with $\lambda I\leq A\leq\Lambda I$ and $w$ satisfies \eqref{main} in $B_{1}^{+}$, which is
\begin{equation*}
    div(x_{n}^{2}A\nabla w)=div(x_{n}^{2}\vec{f}),
\end{equation*}
then there exists $C$ independent of $l$, such that
\begin{equation}
[w]^{*,l+1}_{C_{1}^{1,\alpha}}\leq C\{[f]^{*,l+1}_{C_{0}^{0,\alpha}}+(l+2)[w]^{*,l}_{C_{0}^{0,\alpha}}\}.
\end{equation}
\end{proposition}
\begin{proof}
Let's write $\sigma=[w]^{*,l}_{C_{0}^{0,\alpha}}$ and $\chi=[f]^{*,l+1}_{C_{0}^{0,\alpha}}$. Besides, because of the ``scaling invariance", let's only consider the special case $X=0$, which is
\begin{equation*}
    [\nabla w]_{C^{0}(B_{\frac{1}{l+2}}^{+})}+[\nabla w]_{C^{\alpha}(B_{\frac{1}{l+2}}^{+})}\leq C\{\chi+(l+2)\sigma\}.
\end{equation*}
Other $X\in B_{1}^{+}$'s will be the same, except that we need to care about its scaling. For every $x\in B_{\frac{1}{l+2}}^{+}$, its $\ds\Delta\geq\frac{l+1}{l+2}$. Now let's set $\ds r=\frac{1}{l+2}$. From Lemma \ref{4df} which will be stated after the proof, it then suffices to show for every $x\in B_{\frac{1}{l+2}}^{+}$,
\begin{equation*}
    [\nabla w]_{C^{0}(B_{\frac{1}{2l+4}}^{+}(x))}+[\nabla w]_{C^{\alpha}(B_{\frac{1}{2l+4}}^{+}(x))}\leq C\{\chi+(l+2)\sigma\}.
\end{equation*}
By using a rescaled version of Lemma \ref{centerhigher}, we know
\begin{align*}
    [\nabla w]_{C^{0}(B_{\frac{1}{2l+4}}^{+}(x))}\leq&C(n,\alpha)\{\Delta^{-l-1}[f]^{*,l+1}_{C_{0}^{0,\alpha}}+r^{-1}\Delta^{-l}[w]^{*,l}_{C_{0}^{0,\alpha}}\}\notag\\
    \leq&C\cdot e\{\chi+(l+2)\sigma\}.
\end{align*}
The bound of $[\nabla w]_{C^{\alpha}(B_{\frac{1}{2l+4}}^{+}(x))}$ can be argued similarly.
\end{proof}
\begin{lemma}\label{4df}
For any function $f\in C^{\alpha}(B_{r}^{+}(X))$ with $X\in\mathbb{R}^{n}_{+}$ and $r>0$, we have
\begin{align}
    [f]_{L^{\infty}(B_{r}^{+}(X))}\leq&\sup_{x\in B_{r}^{+}(X)}[f]_{L^{\infty}(B_{r/2}^{+}(x))}\\
    [f]_{C^{\alpha}(B_{r}^{+}(X))}\leq&2^{1-\alpha}\sup_{x\in B_{r}^{+}(X)}[f]_{C^{\alpha}(B_{r/2}^{+}(x))}.
\end{align}
\end{lemma}
\begin{proof}
The first inequality is obvious. To prove the second inequality, we pick any $x_{0},x_{4}\in B_{r}^{+}(X)$. We divide the line segment $\overline{x_{0}x_{4}}$ into four equal sub-segments using the points $\ds x_{i}=\frac{i}{4}x_{4}+\frac{4-i}{4}x_{0}$ with $i\in\{1,2,3\}$. We just need to combine the H\"older bounds in $B_{r/2}^{+}(x_{1})$ and $B_{r/2}^{+}(x_{3})$.
\end{proof}
\subsection{Global non-degenerate Schauder estimate}
Following the same method, we have a parallel result for the uniformly-elliptic equation \eqref{standard}.
\begin{proposition}\label{globaluniformschauder}
Let $l\geq0$ be an arbitrary integer. Assume that $A\in C^{\alpha}$ with $\lambda|\xi|^{2}\leq\xi^{T}A\xi$. If $u$ satisfies \eqref{standard} in $B_{1}^{+}$, which is
\begin{equation*}
    div(A\nabla u)=div(\vec{f}),
\end{equation*}
then there exists $C$ independent of $l$, such that
\begin{equation}
[u]^{*,l+1}_{C_{1}^{1,\alpha}}\leq C\{[f]^{*,l+1}_{C_{0}^{0,\alpha}}+(l+2)[u]^{*,l}_{C_{0}^{0,\alpha}}\}.
\end{equation}
\end{proposition}
\subsection{Majorant power series}
\begin{definition}
For any function $f$ defined in $B_{1}^{+}$ in $x$-coordinate, we define its majorant coefficient and power series as
\begin{equation}\label{power series}
    P_{b}[f]^{(k)}:=[f]^{*,k}_{C_{b}^{k,\alpha}},\ P_{b}[f](t):=\sum_{k=0}^{\infty}P_{b}[f]^{(k)}\frac{t^{k}}{k!}.
\end{equation}
Coefficients are allowed to be $\infty$ and convergence of the power series is not required.
\end{definition}
\begin{remark}
    Here, $P_{b}[f]^{(k)}$ happens to be the $k$-th derivative of $P_{b}[f](t)$, evaluated at $t=0$. Therefore, the abuse of notation $P_{b}[f]^{(k)}$ as $k$-th derivative does not cause confusion.
\end{remark}
\begin{remark}
    As $\Gamma(x)=\Gamma(x')=\Gamma(y')$, we can also define the majorant power series of $\Gamma(x)$ as $P_{b}[\Gamma](t)$.
\end{remark}

Here, we only consider $b=0,1$. In fact, $P_{0}[f](t)\ll P_{1}[f](t)$ for any $f$. If $\partial_{n}f=0$, then $P_{0}[f](t)=P_{1}[f](t)$.

\begin{definition}
Given two power series $f(t)=\sum_{n\geq0}a_{n}t^{n}$ and $g(t)=\sum_{n\geq0}b_{n}t^{n}$ with $b_{n}\in\mathbb{R}\bigcup\{\pm\infty\}$. We say $f$ is majorized by $g$, denoted as $f\ll g$, if $|a_{n}|\leq b_{n}$ for all $n\geq0$.
\end{definition}

The power series $P_{b}[f](t)$ encodes the regularity of $f$ in tangential direction. In fact, if there exists $\delta>0$ such that $P_{b}[f](t)$ is convergent when $|t|\leq\delta$, then $f$ is analytic in tangential direction at $0$, with convergence radius at least $\delta$.

The following Lemma \ref{linearrule} and \ref{composition} provide basic computational rules of the majorant power series. The $C^{\alpha}$ product rule
\begin{equation*}
    [f\cdot g]_{C^{\alpha}(S)}\leq[f]_{L^{\infty}(S)}[g]_{C^{\alpha}(S)}+[f]_{C^{\alpha}(S)}[g]_{L^{\infty}(S)}
\end{equation*}
will be used several times during the proof.

\begin{lemma}\label{linearrule}
Here are some basic properties of majorant power series.
\begin{itemize}
    \item[(a)] (linear) For $f$ and $g$ defined in $B_{1}^{+}$, $P_{b}[f\pm g](t)\ll P_{b}[f](t)+P_{b}[g](t)$.
    \item[(b)] (product) For $f$ and $g$ defined in $B_{1}^{+}$, $P_{b}[f\cdot g](t)\ll P_{b}[f](t)\cdot P_{b}[g](t)$.
    \item[(c)] (integration) Let $f$ be defined on $B_{1}^{+}$ and $g$ be a power series. If $P_{b}[\nabla f](t)\ll g(t)$, then $P_{b}[f](t)\ll t\cdot g(t)+C$ for some $C=C(f,g)$. If in particular $f(0)=0$, then $C\leq g(0)$.
\end{itemize}
\end{lemma}
\begin{proof}
    Part (a) is obvious.
    \begin{itemize}
        \item[(b)] For every $X\in B_{1}^{+}$, we bound its $C^{k}$ semi-norm by
        \begin{align*}
            \Delta^{k}[fg]_{C^{k}(B_{\frac{\Delta}{k+1}}(X))}\leq&\Delta^{k}\sum_{l=0}^{k}\tbinom{k}{l}[f]_{C^{l}(B_{\frac{\Delta}{k+1}})}[g]_{C^{k-l}(B_{\frac{\Delta}{k+1}})}\notag\\
            \leq&\sum_{l=0}^{k}\tbinom{k}{l}\Delta^{l}[f]_{C^{l}(B_{\frac{\Delta}{l+1}})}\Delta^{k-l}[g]_{C^{k-l}(B_{\frac{\Delta}{k-l+1}})}
        \end{align*}
        and bound its $C^{k,\alpha}$ semi-norm by
        \begin{align*}
            \Delta^{k+\alpha}[fg]_{C^{k,\alpha}(B_{\frac{\Delta}{k+1}}(X))}\leq&\Delta^{k+\alpha}\sum_{l=0}^{k}\tbinom{k}{l}\Big\{[f]_{C^{l,\alpha}(B_{\frac{\Delta}{k+1}})}[g]_{C^{k-l}(B_{\frac{\Delta}{k+1}})}\\
            &+[f]_{C^{l}(B_{\frac{\Delta}{k+1}})}[g]_{C^{k-l,\alpha}(B_{\frac{\Delta}{k+1}})}\Big\}\notag\\
            \leq&\sum_{l=0}^{k}\tbinom{k}{l}\Big\{\Delta^{l+\alpha}[f]_{C^{l,\alpha}(B_{\frac{\Delta}{l+1}})}\Delta^{k-l}[g]_{C^{k-l}(B_{\frac{\Delta}{k-l+1}})}\\
            &+\Delta^{l}[f]_{C^{l}(B_{\frac{\Delta}{l+1}})}\Delta^{k-l+\alpha}[g]_{C^{k-l,\alpha}(B_{\frac{\Delta}{k-l+1}})}\Big\}.
        \end{align*}
        By adding these two inequalities together, we have
        \begin{align*}
            &\Delta^{k+\alpha}[fg]_{C^{k,\alpha}(B_{\frac{\Delta}{k+1}}(X))}+\Delta^{k+\alpha}[fg]_{C^{k,\alpha}(B_{\frac{\Delta}{k+1}}(X))}\\
            \leq&\sum_{l=0}^{k}\tbinom{k}{l}\Big\{\Delta^{l}[f]_{C^{l}(B_{\frac{\Delta}{l+1}})}+\Delta^{l+\alpha}[f]_{C^{l,\alpha}(B_{\frac{\Delta}{l+1}})}\Big\}\\
            &\cdot\Big\{\Delta^{k-l}[g]_{C^{k-l}(B_{\frac{\Delta}{k-l+1}})}+\Delta^{k-l+\alpha}[g]_{C^{k-l,\alpha}(B_{\frac{\Delta}{k-l+1}})}\Big\}\\
            \leq&\sum_{l=0}^{k}\tbinom{k}{l}P[f]^{(l)}P[g]^{(k-l)}=(P[f]\cdot P[g])^{(k)}.
        \end{align*}
        \item[(c)] Let $X\in B_{1}^{+}$, then $\Delta^{k}[D^{k}f]_{C^{0}(B^{+}_{\frac{\Delta}{k+1}}(X))}\leq\Delta^{k-1}[D^{k-1}\nabla f]_{C^{0}(B^{+}_{\frac{\Delta}{k}}(X))}$, and similar inequality holds for $C^{0,\alpha}$-global norm. This means for $k\geq1$ we have $P_{b}[f]^{(k)}\leq P_{b}[\nabla f]^{(k-1)}$, so $\ds P_{b}[f](t)\ll\int g+C$. Besides, it's easy to check $\ds\int g\ll t\cdot g(t)$ when $g(t)$ is a positive power series.
    \end{itemize}
\end{proof}
\begin{lemma}\label{composition}
    Let $f$ be defined on $B_{1}^{+}$ and $g$ is a tensor-value function defined on the range $\mathcal{R}$ of $f$. If $[f]_{C^{1}(B_{1})}\leq2$, then their composition satisfies $P_{b}[g\circ f]\ll2\bar{g}\circ P_{b}[f]$, where $\ds\bar{g}\gg\sum_{k=0}^{\infty}\|D^{k}g\|_{C^{\alpha}(\mathcal{R})}\frac{t^{k}}{k!}$ is some scalar-value power series.
\end{lemma}
\begin{proof}
    For any $\varepsilon>0$ (we'll finally let $\varepsilon\searrow0$), there is $X\in B_{1}^{+}$ so that
        \begin{equation*}
            \frac{P_{b}[g\circ f]}{1+\varepsilon}\leq\Delta^{k}[D^{k}(g\circ f)]_{L^{\infty}(B_{\frac{\Delta}{k+1}}^{+}(X))}+\Delta^{k+\alpha}[D^{k}(g\circ f)]_{C^{\alpha}(B_{\frac{\Delta}{k+1}}^{+}(X))}.
        \end{equation*}
        We first consider the first derivative case.
        \begin{align*}
            \frac{P_{b}[g\circ f]}{1+\varepsilon}\leq&\Delta[\nabla g(f)\cdot\nabla f]_{L^{\infty}(B_{\Delta/2}^{+}(X))}+\Delta^{1+\alpha}[\nabla g(f)\cdot\nabla f]_{C^{\alpha}(B_{\Delta/2}^{+}(X))}\notag\\
            \leq&\Delta[g]_{C^{1}}[f]_{C^{1}}+\Delta^{1+\alpha}\{[g]_{C^{1}}[f]_{C^{1,\alpha}}+[g]_{C^{1,\alpha}}[f]_{C^{1}}[f]_{C^{1}}\}\notag\\
            \leq&[g]_{C^{1}}\Delta\{[f]_{C^{1}}+\Delta^{\alpha}[f]_{C^{1,\alpha}}\}+([g]_{C^{1,\alpha}}\Delta[f]_{C^{1}})\cdot[f]_{C^{1}}\notag\\
            \leq&[g]_{C^{1}}P_{b}[f]^{(1)}+([g]_{C^{1,\alpha}}P_{b}[f]^{(1)})\cdot2\leq2\bar{g}'P_{b}[f]^{(1)}.
        \end{align*}
        More generally, higher order chain rule says that there are polynomials $p^{i}$ generated by higher derivatives $Df,D^{2}f,\cdots$ so that
        \begin{equation*}
            D^{k}(g\circ f)=\sum_{i=1}^{k}g^{(i)}(f)\cdot p^{i}.
        \end{equation*}
        From this, we have in $B_{\Delta/2}^{+}(X)$,
        \begin{align*}
            \Delta^{k}[D^{k}(g\circ f)]_{C^{0}}\leq&\sum_{i=1}^{k}\|g^{(i)}\|_{C^{0}}\cdot\Delta^{k-i}\|p^{i}\|_{C^{0}}\notag\\
            \leq&\sum_{i=1}^{k}\|g^{(i)}\|_{C^{0}}\cdot|p^{(i)}(\Delta\|Df\|_{C^{0}},\Delta^{2}\|D^{2}f\|_{C^{0}},\cdots)|\notag\\
            \leq&\sum_{i=1}^{k}\|g^{(i)}\|_{C^{0}}\cdot[p^{(i)}(P[f]^{(1)},P[f]^{(2)},\cdots)]_{C^{0}},
        \end{align*}
        and
        \begin{align*}
            \Delta^{k+\alpha}[D^{k}(g\circ f)]_{C^{\alpha}}\leq&\sum_{i=1}^{k}\|g^{(i)}\|_{C^{0}}\cdot\Delta^{k-i+\alpha}[p^{i}(P[f]^{(1)},P[f]^{(2)},\cdots)]_{C^{\alpha}}\notag\\
            &+\sum_{i=1}^{k}\Delta^{\alpha}[g^{(i)}(f)]_{C^{\alpha}}\cdot\Delta^{k-i}[p^{i}(P[f]^{(1)},P[f]^{(2)},\cdots)]_{C^{0}}.
        \end{align*}
        We similarly bound $\Delta^{\alpha}[g^{(i)}(f)]_{C^{\alpha}}$ by $[g^{(i)}]_{C^{\alpha}}[f]_{C^{1}}\leq2[g^{(i)}]_{C^{\alpha}}$ and computation afterwards is the same as first derivative case.
\end{proof}

\section{Application to regular obstacle problem}

\subsection{Relation to boundary Harnack}
Let $U\geq0$ be a solution of \eqref{U in y abuse} which is better written as
\begin{equation}\label{U in y}
    \partial_{y_{p}}(b^{pq}\partial_{y_{q}}U)=\chi_{\{U>0\}},
\end{equation}
and let $\Gamma$ be the free boundary passing through the origin. As mentioned in the introduction, if $0$ is a regular boundary point, then $y_{n}=\Gamma(y')$ is a Lipschitz graph and boundary points on $\Gamma$ are all regular near $0$. In fact, one can also show that $\Gamma$ is locally $C^{1}$, meaning that after a rotation, $[\Gamma]_{C^{0,1}}\leq\delta$ near the origin for some small $\delta$.

If $b^{pq}=\delta^{pq}$, or equivalently $\Delta U=\chi_{\{U>0\}}$, then classical boundary Harnack ensures that $\ds\frac{U_{y_{i}}}{U_{y_{n}}},\nabla_{y'}\Gamma\in C^{\alpha}$ in $y$-coordinate. More generally, if $b^{pq}\in C^{1}$, then
\begin{equation*}
    \partial_{y_{p}}(b^{pq}\partial_{y_{q}}U_{y_{k}})=-\partial_{y_{p}}(U_{y_{q}}\partial_{y_{k}}b^{pq})=:div(\vec{f}_{k}).
\end{equation*}
in $\{y_{n}>\Gamma(y')\}$. Here, $\vec{f}_{k}$ vanishes at $\Gamma$ and $[\vec{f}]_{C^{\alpha}}$ is small. For such an equation with right-hand-side, one can use the method given by Ros-Oton and Torres-Latorre \cite{RHS} to show that if $[\Gamma]_{C^{0,1}},[\vec{f}_{k}]_{C^{\alpha}}\leq\delta$ for some small $\delta$, then $\ds\frac{U_{y_{i}}}{U_{y_{n}}}\in C^{\alpha}$.

Now let's assume $\Gamma$ is a $C^{1,\alpha}$ graph and $|\nabla_{y'}\Gamma|\leq1$ near $0$. Let's denote
\begin{equation}
    u_{i}:=U_{y_{i}},\quad w_{i}:=\frac{U_{y_{i}}}{U_{y_{n}}},\quad u:=(u_{1},\cdots,u_{n}),\quad w:=(w_{1},\cdots,w_{n}).
\end{equation}
In the region $\{y_{n}>\Gamma(y')\}$, what $u_{k}$ satisfies in $y$-coordinate is
\begin{equation*}
    \partial_{y_{p}}(b^{pq}\partial_{y_{q}}u_{k})=-\partial_{y_{p}}(u_{q}\partial_{y_{k}}b^{pq}),
\end{equation*}
so by Lemma \ref{straightening change} with $f_{p}=-u_{q}\partial_{y_{k}}b^{pq}$, we know that $u_{k}$ satisfies
\begin{equation}\label{equationu}
    \partial_{x_{i}}(a^{ij}\partial_{x_{j}}u_{k})+\partial_{x_{i}}(\frac{\partial x_{i}}{\partial y_{p}}\partial_{y_{k}}b^{pq}\cdot u_{q})=0,\quad a^{ij}=b^{pq}\frac{\partial x_{i}}{\partial y_{p}}\frac{\partial x_{j}}{\partial y_{q}}.
\end{equation}
Here, $\ds\det(\frac{\partial y}{\partial x})=1$ because vertical coordinate change is volume-preserving. The ratio $\ds w_{k}=\frac{u_{k}}{u_{n}}$, by Lemma \ref{ratio}, satisfies
\begin{align}\label{equationw}
    \partial_{x_{i}}(u_{n}^{2}a^{ij}\partial_{x_{j}}w_{k})=&\partial_{x_{i}}(u_{k}u_{q}\frac{\partial x_{i}}{\partial y_{p}}\partial_{y_{n}}b^{pq})-\partial_{x_{i}}(u_{n}u_{q}\frac{\partial x_{i}}{\partial y_{p}}\partial_{y_{k}}b^{pq})\notag\\
    &-u_{q}\frac{\partial x_{i}}{\partial y_{p}}\partial_{y_{n}}b^{pq}\partial_{x_{i}}u_{k}+u_{q}\frac{\partial x_{i}}{\partial y_{p}}\partial_{y_{k}}b^{pq}\partial_{x_{i}}u_{n}.
\end{align}

By implicit function theorem, the ratio $\ds w_{i}=\frac{U_{y_{i}}}{U_{y_{n}}}$ represents the slope of $U$'s level set, so $\ds\partial_{y_{k}}\Gamma=\lim_{y_{n}\searrow\Gamma(y')}w_{k}$.
A naive application of Theorem \ref{HOBH straight} can inductively imply $\Gamma\in C^{\infty}(B_{1/2})$. Now let's be more careful in order to prove $\Gamma\in C^{\omega}$.

\subsection{A PDE system}
The expression \eqref{equationu} and \eqref{equationw} depend on $u$, $w$, $a^{ij}$, $\ds\frac{\partial x}{\partial y}$, $\nabla_{y}B$, which in turn depend on the free boundary $\Gamma$. Two important quantities are $\nabla u$ and $w$, and we define their corresponding power series:
\begin{equation}
    \Pi(t)=P_{0}[\nabla u](t),\quad\Omega(t)=P_{1}[w](t).
\end{equation}

Next we write \eqref{equationu} and \eqref{equationw} in a more convenient form. Let's write
\begin{equation}
    \tilde{A}=(\frac{u_{n}}{x_{n}})^{2}A,
\end{equation}
which is uniformly elliptic by Hopf lemma, and
\begin{align}
    \vec{F}=&(\frac{u_{k}}{x_{n}}\frac{u_{q}}{x_{n}}\frac{\partial x_{i}}{\partial y_{p}}\partial_{y_{n}}b^{pq}-\frac{u_{n}}{x_{n}}\frac{u_{q}}{x_{n}}\frac{\partial x_{i}}{\partial y_{p}}\partial_{y_{k}}b^{pq})e_{i},\\
    G=&\frac{u_{q}}{x_{n}}\frac{\partial x_{i}}{\partial y_{p}}\partial_{y_{k}}b^{pq}\partial_{x_{i}}u_{n}-\frac{u_{q}}{x_{n}}\frac{\partial x_{i}}{\partial y_{p}}\partial_{y_{n}}b^{pq}\partial_{x_{i}}u_{k},\\
    \vec{H}=&\frac{\partial x_{i}}{\partial y_{p}}\partial_{y_{k}}b^{pq}\cdot u_{q}e_{i}.
\end{align}
With these notations, $A,\tilde{A},\vec{F},G,\vec{H}$ are polynomials in the variables
\begin{equation*}
    \Big(B,\nabla_{y}B,\frac{\partial x}{\partial y},\nabla u,\frac{u}{x_{n}}\Big),
\end{equation*}
and \eqref{equationw} and \eqref{equationu} become the following PDE system:
\begin{align}
    div(x_{n}^{2}\widetilde{A}\cdot\nabla w)=&div(x_{n}^{2}\vec{F})+x_{n}G,\label{system}\\
    div(A\cdot\nabla u)=&div(\vec{H}).\label{system2}
\end{align}
The following lemma describes how the coefficients $A,\tilde{A},\vec{F},G,\vec{H}$ are dominated by $\Pi(t)$ and $\Omega(t)$.
\begin{lemma}\label{AAFGH}
Assume that $|\nabla\Gamma|\leq1$ and the matrix $B$ is analytic in $y$-coordinate with large convergence radius. There exist analytic functions
\begin{equation*}
\mathcal{A},\widetilde{\mathcal{A}},\mathcal{F},\mathcal{G},\mathcal{H}
\end{equation*}
in the variables $(t,\Pi,\Omega)$ with large convergence radius so that
\begin{equation*}
    P_{0}[A]\ll\mathcal{A}(t,\Pi(t),\Omega(t)),\quad P_{0}[\widetilde{A}]\ll\widetilde{\mathcal{A}}(t,\Pi(t),\Omega(t)),\quad\mbox{etc}.
\end{equation*}
\end{lemma}
\begin{proof}
    The majorant inequalities follow from Lemma \ref{geometry} and \ref{reducegamma} below. Lemma \ref{geometry} yields that $A,\tilde{A},\vec{F},G,\vec{H}$ are analytic functions in the variables
\begin{equation*}
    \Big(x,\Gamma,\nabla\Gamma,\nabla u,\frac{u}{x_{n}}\Big).
\end{equation*}
Lemma \ref{reducegamma} implies that $P_{0}[\Gamma]$,$P_{0}[\nabla\Gamma]$ are controlled by $\Omega(t)$, and $\ds P_{0}[\frac{u}{x_{n}}]$ is controlled by $\Pi(t)$.
\end{proof}

\begin{lemma}\label{geometry}
Assume that $B=(b^{pq})$ is analytic in $y$-coordinate with large convergence radius, and that $\Gamma$ is passing through $0$ with $|\nabla_{y'}\Gamma|\leq1$, then after the $x-y$ coordinate change,
\begin{itemize}
    \item[(a)] $P_{1}[y(x)]\ll2+t+P_{0}[\Gamma]$ and $\ds P_{1}[\frac{\partial y}{\partial x}],P_{1}[\frac{\partial x}{\partial y}]\ll1+P_{0}[\nabla\Gamma]$.
    \item[(b)] If we treat $B(x)=B(y(x))$ and $\nabla_{y}B(x)=\nabla_{y}B(y(x))$ as functions in $x$-coordinate, then there is some large $R$ so that
    \begin{equation}
        P_{1}[B],P_{1}[\nabla_{y}B]\ll\frac{C}{R-t-P_{0}[\Gamma]}.
    \end{equation}
    \item[(c)] For $\ds a^{ij}=b^{pq}\frac{\partial x_{i}}{\partial y_{p}}\frac{\partial x_{j}}{\partial y_{q}}$, there is some large $R$ so that
    \begin{equation}
        P_{0}[A]\ll C\frac{(1+P_{0}[\nabla\Gamma])^{2}}{R-t-P_{0}[\Gamma]}.
    \end{equation}
\end{itemize}
\end{lemma}
\begin{proof}
We only prove (a),(b), and (c) is just obtained from (a),(b).
    \begin{itemize}
        \item[(a)] This is because $y=x+\Gamma(x)e_{n}$ (see \eqref{xtoy}), and $\ds\frac{\partial y}{\partial x}=I_{n\times n}+\nabla\Gamma\otimes e_{n}$.
        \item[(b)] As we assume $|\nabla_{y'}\Gamma|\leq1$, it follows that
        \begin{equation*}
            [y(x)]_{C^{1}}\leq[x]_{C^{1}}+[\Gamma(x')e_{n}]_{C^{1}}\leq2.
        \end{equation*}
        This means we can apply the composition rule Lemma \ref{composition}.
        
    \end{itemize}
\end{proof}

\begin{lemma}\label{reducegamma}
Assume that $\Gamma(0)=0$ and $\ds u\Big|_{B_{1}'}=0$, then
\begin{itemize}
    \item[(a)] $P_{0}[\nabla\Gamma]=P_{1}[\nabla\Gamma]\ll P_{1}[w]$,
    \item[(b)]$P_{0}[\Gamma](t)\ll P_{1}[w]^{(0)}+t\cdot P_{1}[w](t)$,
    \item[(c)] $\ds P_{0}[\frac{u}{x_{n}}]\ll P_{0}[\nabla u]$
    \item[(d)]$P_{0}[u]\ll(1+t)\cdot P_{0}[\nabla u](t)$.
\end{itemize}
\end{lemma}
\begin{proof}
\begin{itemize}
    \item[(a)] By implicit function theorem, $\ds\partial_{k}\Gamma=\lim_{x_{n}\to0}\frac{U_{y_{k}}}{U_{y_{n}}}=\lim_{x_{n}\to0}w_{k}$, so $P_{1}[\nabla\Gamma]\ll P_{1}[w]$. An application of integration rule in Lemma \ref{linearrule} hence proves (b).
    \item[(c)] When $\ds u\Big|_{B_{1}'}=0$, then $\ds\frac{u}{x_{n}}=\int_{0}^{1}e_{n}\cdot\nabla u(x',t x_{n})dt$. By taking higher order derivatives to this integral, we obtain (c). We can then use the integration or product rule in Lemma \ref{linearrule} to prove (d).
\end{itemize}
\end{proof}

\subsection{Analyticity of free boundary}
Now let's inductively estimate the higher regularity of $w$ and $u$ using an ODE system. For convenience, we will treat analytic expressions like $\mathcal{F}(t,\Pi(t),\Omega(t))$ as power series of $t$ and denote
\begin{equation*}
    \mathcal{F}^{(k)}:=\frac{d^{k}}{dt^{k}}\Big|_{t=0}\mathcal{F}(t,\Pi(t),\Omega(t)).
\end{equation*}
With this notation, we have $P_{0}[\vec{F}]^{(k)}\leq\mathcal{F}^{(k)}$, $P_{0}[\widetilde{A}]^{(k)}\leq\widetilde{\mathcal{A}}^{(k)}$, e.t.c..

First, we have an estimate of higher derivatives of $w$.
\begin{lemma}
Assume that $div(x_{n}^{2}\tilde{A}\cdot\nabla w)=div(x_{n}^{2}\vec{F})+x_{n}G$ where $\tilde{A}$ is uniformly elliptic, and
\begin{equation*}
    P_{0}[\widetilde{A}]\ll\widetilde{\mathcal{A}}(t,\Pi,\Omega),\quad P_{0}[\vec{F}]\ll\mathcal{F}(t,\Pi,\Omega),\quad P_{0}[G]\ll\mathcal{G}(t,\Pi,\Omega),
\end{equation*}
then there is analytic and positive power series $\mathcal{M}$ so that
\begin{equation}
    \frac{d}{dt}\Omega\ll\mathcal{M}(t,\Pi,\Omega).
\end{equation}
\end{lemma}
\begin{proof}
Let $\beta$ be any tangential multi-index with $|\beta|=k\geq1$, applying $D^{\beta}$ to \eqref{system} gives
\begin{align*}
    div(x_{n}^{2}\widetilde{A}\cdot\nabla D^{\beta}w)=&div(x_{n}^{2}D^{\beta}\vec{F})+x_{n}D^{\beta}G\notag\\
    &-\sum_{0<\gamma\leq\beta}\prod_{m=1}^{n}\tbinom{\beta_{m}}{\gamma_{m}}\partial_{x_{i}}[x_{n}^{2}(D^{\gamma}\widetilde{a}^{ij})(\partial_{x_{j}}D^{\beta-\gamma}w_{k})].
\end{align*}
The global Schauder estimate proposition \ref{globalschauder} yields that
\begin{align*}
    \Omega^{(k+1)}\leq&C\{(k+2)[w]^{*,k}_{C_{0}^{k,\alpha}}+[\vec{F}]^{*,k+1}_{C_{0}^{k,\alpha}}+[G]^{*,k+1}_{C_{0}^{k,\alpha}}+\sum_{l=1}^{k}\tbinom{k}{l}[\widetilde{A}]^{*,l}_{C_{0}^{l,\alpha}}[w]^{*,k-l+1}_{C_{1}^{l,\alpha}}\}\notag\\
    \leq&C\{(k+2)\Omega^{(k)}+\mathcal{F}^{(k)}+\mathcal{G}^{(k)}+\sum_{l=1}^{k}\tbinom{k}{l}\widetilde{\mathcal{A}}^{(l)}\Omega^{(k-l+1)}\}\notag\\
    \leq&C\{(t\frac{d}{dt}\Omega)^{(k)}+(\Omega+\mathcal{F}+\mathcal{G})^{(k)}+[(\widetilde{\mathcal{A}}-\widetilde{\mathcal{A}}^{(0)})\frac{d}{dt}\Omega]^{(k)}\}.
\end{align*}
Here, we have used a combinatoric identity
\begin{equation*}
\sum_{\substack{\gamma\leq\beta\\|\gamma|=l}}\prod_{m=1}^{n}\tbinom{\beta_{m}}{\gamma_{m}}=\tbinom{|\beta|}{l}.
\end{equation*}
By setting $k\geq0$ to be arbitrary, we have
\begin{equation*}
    \frac{d}{dt}\Omega\ll C\{t\frac{d}{dt}\Omega+\Omega+\mathcal{F}+\mathcal{G}+(\widetilde{\mathcal{A}}-\widetilde{\mathcal{A}}^{(0)})\frac{d}{dt}\Omega\}.
\end{equation*}
Therefore, we obtain
\begin{equation}
    \frac{d}{dt}\Omega\ll C\frac{\Omega+\mathcal{F}+\mathcal{G}}{1-t-(\widetilde{\mathcal{A}}-\widetilde{\mathcal{A}}^{(0)})}=:\mathcal{M}(t,\Pi,\Omega).
\end{equation}
We remark that $(\widetilde{\mathcal{A}}-\widetilde{\mathcal{A}}^{(0)})$ is a positive power series starting from linear term.
\end{proof}

Similarly, we have an estimate of higher derivatives of $u$.
\begin{lemma}
Assume that $div(A\cdot\nabla u)=div(\vec{H})$ where $A$ is uniformly elliptic, and
\begin{equation*}
    P_{0}[A]\ll\mathcal{A}(t,\Pi,\Omega),\quad P_{0}[\vec{H}]\ll\mathcal{H}(t,\Pi,\Omega),
\end{equation*}
then there is analytic and positive power series $\mathcal{N}$ so that
\begin{equation}
    \frac{d}{dt}\Pi\ll\mathcal{N}(t,\Pi,\Omega).
\end{equation}
\end{lemma}

With these the standard Cauchy-Kovalevskaya theorem for ODEs proves the analyticity.

\begin{corollary}
    If $\Gamma(0)=0$, $\|\nabla_{y'}\Gamma\|_{C^{\alpha}}\leq1$, and $\Pi^{(0)}$ and $\Omega^{(0)}$ are bounded, then $\Pi(t)$ and $\Omega(t)$ are analytic near $t=0$, with converging radius uniformly bounded from below.
\end{corollary}
\begin{proof}
    As $\Pi$ and $\Omega$ satisfy two ODE inequalities, they are majorized by ODE solutions of the same right hand side and the same initial data, because the right hand sides are positive power series. Analyticity of first order ODE system gives a uniform lower bound of the convergence radius of $\Pi(t)$ and $\Omega(t)$.
\end{proof}

{\it Proof of Theorem \ref{analytic}}. Analyticity of $w$ in $x$-coordinate implies the analyticity of $\Gamma(x)$ in $x$-coordinate. As $\Gamma(x)=\Gamma(x')$, and $x'=y'$ in the $x-y$ coordinate change, $\Gamma(y')$ is also analytic in $y$-coordinate.
\qed

\end{document}